\documentclass[a4paper,12pt]{article}

\usepackage{geometry}
\usepackage{graphicx}
\usepackage{amssymb}
\usepackage{amsmath}
\usepackage{amsthm}
\usepackage{amsfonts}
\usepackage{epsfig}
\usepackage{amsmath,amsfonts, amssymb, amsthm, amscd,graphicx}
\usepackage{latexsym}
\usepackage{epsfig,epsf,psfrag}
\usepackage{nonfloat}

\usepackage{amsmath}
\usepackage{amsthm}
\usepackage{amssymb}
\usepackage{cancel}

\usepackage{epsfig}
\usepackage{psfrag}
\usepackage{subfigure}
\usepackage{graphicx}

\usepackage[mathscr,mathcal]{eucal}
\usepackage{enumerate}

\usepackage{t1enc}
\usepackage[latin2]{inputenc}

\def \C {\ensuremath{\mathbb C}}

\def \R {\ensuremath{\mathbb R}}

\def \Z {\ensuremath{\mathbb Z}}

\def\cF{\mathcal{F}}

\newcommand{\prob}[1]{\ensuremath{\mathbf{P}\left(#1\right)}}
\newcommand{\expect}[1]{\ensuremath{\mathbf{E}\left(#1\right)}}

\newcommand{\condprob}[2]{\ensuremath{\mathbf{P}\big(#1\bigm|#2\big)}}
\newcommand{\condexpect}[2]{\ensuremath{\mathbf{E}\big(#1\bigm|#2\big)}}

\newcommand{\ind}[1]{\ensuremath{{1\!\!1}{\{#1\}}}}

\def\pt{\partial_t}

\def\px{\partial_x}

\def\ph{\partial_h}

\newcommand{\abs}[1]{\left|{#1}\right|}

\newtheorem {theorem}{Theorem}

\newtheorem {lemma}{Lemma}

\newtheorem {corollary}{Corollary}

\newtheorem {proposition}{Proposition}

\newtheorem* {theorem*}{Theorem}
\newtheorem* {thm*}{Theorem}
\newtheorem* {lemma*}{Lemma}
\newtheorem* {lem*}{Lemma}
\newtheorem* {corollary*}{Corollary}
\newtheorem* {cor*}{Corollary}
\newtheorem* {proposition*}{Proposition}
\newtheorem* {prop*}{Proposition}
\newtheorem* {definition*}{Definition}
\newtheorem* {def*}{Definition}
\newtheorem* {conjecture*}{Conjecture}
\newtheorem* {remark*}{Remark}
\newtheorem* {rem*}{Remark}

\theoremstyle{definition}

\newtheorem*{remark}{Remark}

\def\be{\begin{equation}}
\def\ee{\end{equation}}

\def\bea{\begin{eqnarray}}
\def\eea{\end{eqnarray}}

\def \Ai {\ensuremath{\mbox{Ai}}}

\title{Marginal densities of the ``true'' self-repelling motion}
\author{{\sc Laure Dumaz} \qquad and \qquad {\sc B\'alint T\'oth}}

\begin{document}

\maketitle

\begin{abstract}
Let $X(t)$ be the \emph{true self-repelling motion (TSRM)} constructed in \cite{toth_werner_98}, $L(t,x)$ its occupation time density (local time) and $H(t):=L(t,X(t))$ the height of the local time profile at the actual position of the motion. The joint distribution of $(X(t),H(t))$ was identified in \cite{toth_95} in somewhat implicit terms. Now we give explicit formulas for the densities of the marginal distributions of $X(t)$ and $H(t)$. The distribution of $X(t)$ has a particularly surprising shape (see Picture (\ref{fignu1}) below): It has a sharp local \emph{minimum} with discontinuous derivative at $0$. As a consequence we also obtain a precise version of the large deviation estimate of \cite{dumaz_11}.
\end{abstract}

\begin{figure}[!h]
\begin{minipage}[b]{7cm}
\centering\epsfig{figure=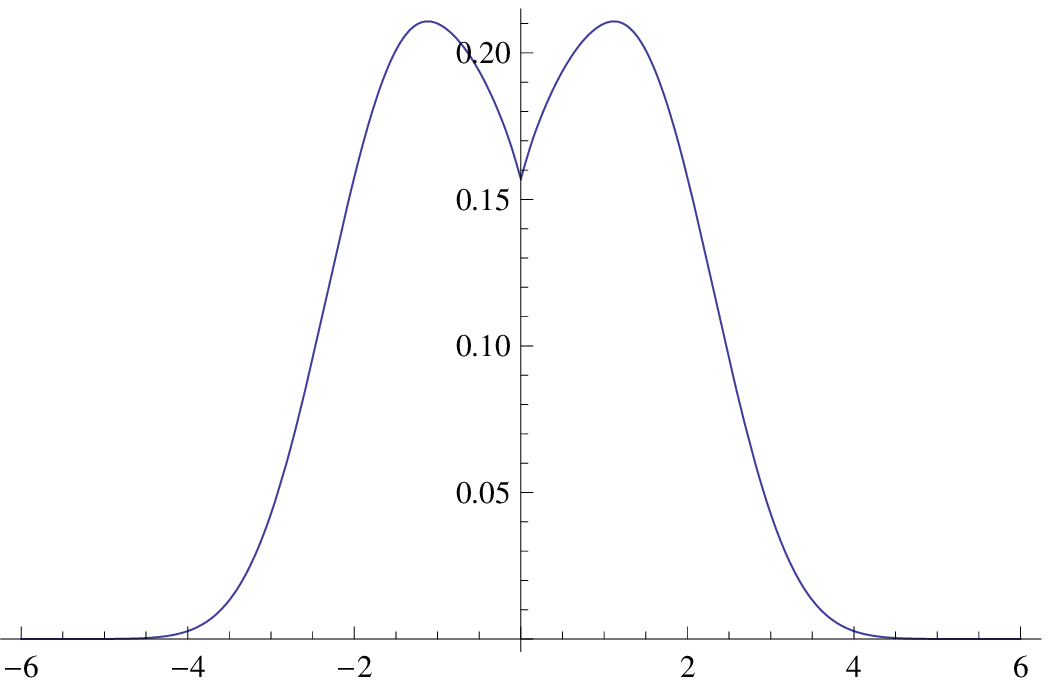,width=7cm}
\caption{Density of $X(1)$
\label{fignu1}}
\end{minipage} \hfill
\begin{minipage}[b]{7cm}
\centering\epsfig{figure=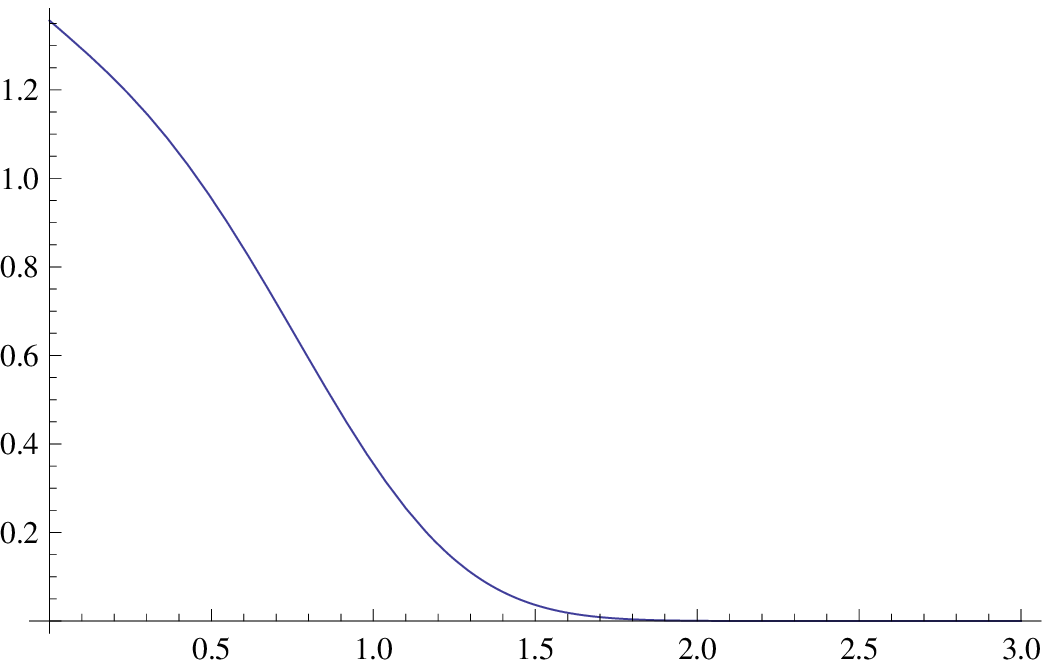,width=7cm}
\caption{Density of $H(1)$
\label{fignu2}}
\end{minipage}
\end{figure}

\section{Introduction and main results}
\label{s:intro}

\subsection{Introduction}
\label{ss:intro}

In the present paper, we study some marginal densities of a self-interacting one dimensional process
called the true self-repelling motion (TSRM), defined in \cite{toth_werner_98}. The true self-repelling motion is a continuous real-valued process $(X(t),\; t \ge 0)$ that is
locally self-interacting with its past occupation-time: It is pushed away from the places it has visited the most in its past. This process has a completely singular behavior compared to diffusions: For example, it has a finite variation of order 3/2, which contrasts with the finite quadratic variation of semi-martingales.

A crucial property of the TSRM is the existence of a continuous density for its occupation-time measure, denoted $L(t,x)$ and named again ``local time'' by analogy with semi-martingales (see Theorem $4.2$ of \cite{toth_werner_98}). This enables to prove that the TSRM is driven by the negative gradient of its own occupation time density $L(t,x)$. Written \emph{formally}:
\begin{align}
\label{tsrmdiffeq}
&
\pt X(t) = -\px L(t,X(t)),
&&
\pt L(t,x) = \delta(x-X(t))
\\[6pt]
\label{tsrm0ic}
&
X(0)=0,
&&
L(0,x)=0.
\end{align}
The driving mechanism \eqref{tsrmdiffeq} is to be considered in a properly regularized form. Making proper mathematical sense of this mechanism is the main content of \cite{toth_werner_98}.

The local time at the current position of the TSRM is denoted by
\begin{align*}
H(t) := L(t,X(t))
\end{align*}
and is called the \emph{height process}.
The processes $X(t)$, the local time $L(t,x)$ and the height $H(t)$ scale as follows:
\[
X(at) \sim a^{2/3} X(t),
\qquad
L(at, a^{2/3}x) \sim a^{1/3}L(t,x),
\qquad
H(at) \sim a^{1/3} H(t).
\]
Throughout this paper, the notation $\sim$ means equality in distribution between the two terms.

Scaling and many other features of TSRM and related objects show similarities with the distributions appearing in the KPZ (Kardar-Parisi-Zhang) universality class. The computations of the present paper also show resemblances with KPZ-related computations (e.g. the natural appearance of Airy functions in both contexts). As shown in \cite{warren_07} and \cite{tribe_zaboronski_11} determinental and Pfaffian structures also appear in the Brownian Web (which is in the background of the construction of TSRM) just like in models belonging to the KPZ universality class. However, it is not yet clear whether there exist some deeper connections between TSRM and KPZ.

We will examine in this paper the distributions of the marginals $X(t)$ and $H(t)$ at the fixed time $t$. It turns out that it is more convenient to work at first with an exponential random time of parameter $s$ independent of the process instead of with a fixed time.
With Feynman-Kac formulas applied for the Brownian motion, it is possible to deduce explicit expressions for its marginals (in terms of Airy function, confluent hypergeometric functions and Mittag-Leffler density). It leads to Pictures \ref{fignu1} and \ref{fignu2} for $t=1$. We found the shape of the density of the position $X(t)$ rather unusual and surprizing. Indeed, it  has a {\it sharp wedge-like local
minimum} at $x=0$ with discontinuous derivative, and global maxima
away from zero. It means that at any positive time, the process still
remembers its starting point and it is strongly pushed away from it.
 These results are also confirmed by numerical simulations (see Figures \ref{histX}. and \ref{histH}.).

\subsection{Review of background and notations}
\label{ss:review}

The \emph{true self-avoiding walk (TSAW)} is a nearest neighbor random walk with long memory on $\Z^d$ pushed by the negative gradient of its own occupation time measure (local time). For historical background see \cite{amit_parisi_peliti_83}, \cite{obukhov_peliti_83}, \cite{peliti_pietronero_87} or the introductions to \cite{toth_95} or \cite{toth_werner_98}. In the present paper we are interested only in the one-dimensional case.

A discrete time version considered and analyzed in \cite{toth_95} is the following. Let $n\mapsto S(n)\in\Z$ be a nearest neighbor walk on $\Z$. The occupation time on lattice bonds, lattice sites, and the negative (discrete) gradient of bond occupation time are in turn:
\begin{align*}
&
\ell(n,j+1/2):=
\#\{0\le m< n: (S(m)+S(m+1))/2=j+1/2\},
\\[6pt]
&
\ell(n,j):=
\big(\ell(n,j-1/2)+\ell(n,j+1/2)\big)/2,
\\[6pt]
&
\delta(n,j):=
\ell(n,j-1/2)-\ell(n,j+1/2),
\end{align*}
where $j\in\Z$ label lattice sites, $j\pm\frac12\in\Z+\frac12$ label lattice bonds and $n\in\Z_+$ is the (discrete) time of the process.

The TSAW (with bond repulsion) is defined by the law:
\[
\condprob{S(n+1)=j\pm1}{\cF_n, S(n)=j}=
\frac{w(\pm\delta(n,j))}{w(\delta(n,j))+w(-\delta(n,j))},
\]
where $w:\Z\to(0,\infty)$ is a fixed rate function which is assumed nondecreasing and non-constant. In \cite{toth_95} limit theorem was proved for the distribution of the displacement of the TSAW, with time-to-the-two-thirds scaling. Later, the convergence was established for the whole process \cite{toth_werner_98, newman_04}. In order to expose this and also prepare our present work we need some notation and preliminary concepts.

\medskip

\noindent Let $\R\ni y\mapsto B(y)\in\R_+$ be a \emph{two sided Brownian motion}, starting from level $h\ge0$, and let $\omega'$ be the first hitting time of zero by the backward Brownian motion, $\omega_x$ the first one after time $x \ge 0$ of the forward Brownian motion:
\begin{align}
\label{defw'}
&
\omega':=\sup\{y\le0: B(y)=0\},
\\[6pt]
\label{defwx}
&
\omega_{x}:=\inf\{y\ge x: B(y)=0\}, \qquad x \ge 0,
\end{align}
and denote:
\begin{align}
\label{defTx}
T_{x} = \int_{\omega'}^{\omega_{x}} |B(t)| dt.
\end{align}

The random variable $T_{x}$ is a.s. positive and finite and has an absolutely continuous distribution with density
\[
\varrho(t;x,h):=
\pt \condprob{T_{x}<t}{B(0) = h},
\qquad
t\in\R_+, \,\, x\in\R_+, \,\, h\in\R_+.
\]
The Laplace transform of this density with respect to the variable $t$ is
\begin{align}
\label{rholap}
\hat\varrho(s;x,h)
:=
s\int_0^\infty e^{-st}\varrho(t;x,h) dt
=
s \condexpect{e^{-sT_{x}}}{B(0) = h}.
\end{align}
We extend $\varrho$ and $\hat \varrho$ to $x \in \R$ as \emph{even} functions: $\varrho(t;x,h) = \varrho(t;|x|,h)$ and $\hat \varrho(s;x,h) = \hat \varrho(s;|x|,h)$ for all $x \in \R$, $h \ge 0$, $s\ge 0$ and $t\ge 0$.  We will see soon that $\varrho(t;x,h)$ (respectively, $\hat \varrho(s;x,h)$)  viewed as functions of the variables $(x,h)$ turn out to be the density of of the joint distribution of $(X_t,H_t)$  (respectively, $(X_{\theta_s},H_{\theta_s})$ for $\theta_s$ independent exponentially distributed random variable of expectation  $s^{-1}$). Note that \emph{a priori} is not at all evident that those functions define indeed densities on $\R \times \R^+$, see \eqref{intnu}.

We introduce the following marginals:
\begin{align*}
&
\varrho_1(t;x):=\int_{0}^\infty \varrho(t;x,h)dh,
\quad
&&
\varrho_2(t;h):=\int_{-\infty}^\infty \varrho(t;x,h)dx,
\\[6pt]
&
\hat\varrho_1(s;x):=\int_{0}^\infty \hat\varrho(s;x,h)dh,
\quad
&&
\hat\varrho_2(s;h):=\int_{-\infty}^\infty \hat\varrho(s;x,h)dx.
\end{align*}
Brownian scaling
$\big(B(a\cdot)\,\big|\,B(0)=\sqrt a h\big) \sim
\big(\sqrt a B(\cdot)\,\big|\,B(0)= h \big)$ implies the following scaling relations:
\begin{align}
\label{rhoscale}
&
\varrho(t;x,h)= t^{-1} \nu(t^{-2/3}x,t^{-1/3}h),
&&
\varrho_1(t;x)= t^{-2/3} \nu_1(t^{-2/3}x),
&&
\varrho_2(t;h)= t^{-1/3} \nu_2(t^{-1/3}h),
\\[6pt]
\notag
&
\hat\varrho(s;x,h)= s \hat \nu(s^{2/3}x,s^{1/3}h),
&&
\hat\varrho_1(s;x)= s^{2/3} \hat\nu_1(s^{2/3}x),
&&
\hat\varrho_2(s;h)= s^{1/3} \hat\nu_2(s^{1/3}h),
\end{align}
where
\begin{align*}
&
\nu(x,h):=\varrho(1;x,h),
&&
\nu_1(x):=\varrho_1(1;x),
&&
\nu_2(h):=\varrho_2(1;h),
\\[6pt]
&
\hat\nu(x,h):=\hat\varrho(1;x,h),
&&
\hat\nu_1(x):=\hat\varrho_1(1;x),
&&
\hat\nu_2(h):=\hat\varrho_2(1;h).
\end{align*}

We recall the main results of \cite{toth_95} and \cite{toth_werner_98} and more recent large deviation estimates from \cite{dumaz_11} which are necessary background material for the present paper:

\begin{itemize}

\item
Theorem 2 of \cite{toth_95}:
$(x,h)\mapsto\nu(x,h)$ and $(x,h)\mapsto\hat\nu(x,h)$ are probability densities on $\R\times\R_+$:
\begin{align}
\label{intnu}
\int_{-\infty}^\infty\int_0^\infty \nu(x,h)dhdx
=1=
\int_{-\infty}^\infty\int_0^\infty \hat\nu(x,h)dhdx.
\end{align}

\item
Theorem 3 of \cite{toth_95}:
Let $s>0$ be fixed and $\vartheta_{n}$ a sequence of geometrically distributed stopping times with $\expect{\vartheta_{n}}=n/s$, independent of the TSAW $S(\cdot)$. The following limit theorem holds:
\begin{align}
\label{limit}
\big(
(\alpha n)^{-2/3}S(\vartheta_n),
(\alpha n)^{-1/3}\ell(\vartheta_n, S(\vartheta_n))
\big)
\Rightarrow
\big(\hat X_s, \hat H_s\big),
\end{align}
where the density of the joint distribution of $\big(\hat X_s, \hat H_s\big)$ is: 
\begin{align*}
\partial^2_{xh} \prob{\hat X_s<x, \hat H_s)<h}
=
\hat\varrho(s;x,h).
\end{align*}
The constant $\alpha$ in the norming factors on the left hand side of \eqref{limit} depend only on the rate function $w(\cdot)$ and is given by formula (1.23) in \cite{toth_95}.
\\[6pt]
Actually, in \cite{toth_95} only the case $w(z)=e^{-\beta z}$ was considered and only the limit theorem for the displacement was stated explicitly. However, the proofs implicitly contain these extensions.

\item
Remark after Theorem 3 of \cite{toth_95}: Assume that the sequence of random variables $\big(n^{-2/3}S(n), n^{-1/3}\ell(n, S(n))\big)\in\R\times\R_+$ is tight.  Let $t>0$ be fixed. Then
\[
\big(
(\alpha n)^{-2/3}S([nt]),
(\alpha n)^{-1/3}\ell([nt], S([nt]))
\big)
\Rightarrow
\big(X(t), H(t)\big),
\]
where the density of the joint distribution of $\big(X(t), H(t)\big)$ is:
\begin{align}
\label{jointdensity}
\partial^2_{xh} \prob{X(t)<x, H(t)<h}
=
\varrho(t;x,h).
\end{align}
\medskip
See also \cite{toth_veto_11} for similar results for a continuous time version with on-site-repulsion.

\item
Finally, in \cite{toth_werner_98} the \emph{true self-repelling motion (TSRM)} mentioned in the first paragraph of this paper, was constructed. This is the scaling limit of the properly rescaled TSAW, $t\mapsto(\alpha n)^{-2/3} S([nt])$, as $n\to\infty$. Actually, this construction goes "backwards". First the three-parameter process $\{\Lambda_{x,h}(y): x\in\R, \ h\in\R_+, \ y\in\R\}$ is constructed, which later got the catchy name \emph{Brownian Web}, \cite{fontes_isopi_newman_ravishankar_04}. Then, from $\Lambda_{\cdot,\cdot}(\cdot)$, by an \emph{inverse Ray-Knight procedure} the process $t\mapsto(X(t),H(t))\in\R\times\R_+$ is built. From this construction it follows that $t\mapsto X(t)$ has a regular local time (i.e. occupation time density) process $\{L(t,x): t\in\R_+, x\in\R\}$, and the Brownian Web $\Lambda_{\cdot,\cdot}(\cdot)$ is identified as the local time profile of $X(\cdot)$ stopped at inverse local times $\tau_{x,h}:=\inf\{t: L(t,x)\ge h\}$:
\[
\Lambda_{x,h}(y)=L(\tau_{x,h}, y), 
\qquad
H(t)=L(t,X(t)). 
\]
From this inverse Ray-Knight construction the identity \eqref{jointdensity} also drops out.

\item
By choosing the initial conditions
\begin{align}
\notag
&
X(0)=0,
&&
L(0,x)=B(x),
\end{align}
where $x\mapsto B(x)$ is a two sided 1d Brownian motion, rather than \eqref{tsrm0ic}, we obtain a version of the TSRM \emph{with stationary increments}. Denote this process $X_{\mathrm{st}}(t)$, its local time $L_{\mathrm{st}}(t,x)$, and the height process $H_{\mathrm{st}}(t):= L_{\mathrm{st}}(t,X_{\mathrm{st}}(t))$. The scaling of the stationary version is similar to that of the TSRM starting from initial conditions \eqref{tsrm0ic}:
\[
X_{\mathrm{st}}(at) \sim a^{2/3} X_{\mathrm{st}}(t),
\quad
L_{\mathrm{st}}(at, a^{2/3}x) \sim a^{1/3}L_{\mathrm{st}}(t,x),
\quad
H_{\mathrm{st}}(at) \sim a^{1/3} H_{\mathrm{st}}(t).
\]
Moreover, the following equality in distribution holds: 
\begin{align}
\label{xstatsimx}
X_{\mathrm{st}}(1) \sim 2^{-1/3} X(1).
\end{align}
This identity in law follows directly from the construction of $t\mapsto (X_{\mathrm{st}}(t), H_{\mathrm{st}}(t))$ outlined in section 10. of \cite{toth_werner_98}. The main point is that having the bottom line of the local time profiles a Brownian motion $B(x)$ rather identically $0$ changes the distribution of the \emph{one-dimensional marginal distributions} of the displacement by a multiplicative factor only. However, we also emphasize that this is valid only for the one-dimensional marginals of $X_{\mathrm{st}}(t)$ and in particular  no distributional identity of this sort holds for the 1d marginals $H_{\mathrm{st}}(1)$.

\item
In \cite{dumaz_11}, using TSRM's construction, asymptotics for the tails of $X(1)$ and $H(1)$ are found. It is shown that:
\begin{align*}
&
\prob{X(1)>x} = \exp\left(- (4 \delta'^3_1 x^3)/27 + O(\ln(x))\right),
\\[6pt]
\notag
&
\prob{H(1)>h} = \exp\left(-(8 h^{3})/9 + O(\ln(h))\right),
\end{align*}
where $\delta'_1$ is the absolute value of the first negative zero of the derivative of a renormalized Airy function. It is precisely defined in paragraph \ref{ss:appAiry}.

In the stationary case, the tails of the distributions are also computed. In this case, due to \eqref{xstatsimx}, the constant in front of the $x^3$ is multiplied by $2$ for $X_{\mathrm{st}}(1)$. The height $H_{\mathrm{st}}(1)$ behaves differently as the initial local time helps to have large values of $|H_{\mathrm{st}}(1)|$. It is shown that there exists $\eta >0$ such that for every large $h$:
\begin{align*}
\exp\left(- 1/\eta \,h^{3/2}\right)\le \prob{\pm H_{\mathrm{st}}(1)>h} \le \exp\left(-\eta\, h^{3/2}\right).
\end{align*}

\end{itemize}

\subsection{Main results}
\label{ss:results}

In the present paper, our main goal is to identify the marginal distributions, $\nu_1(x)$ and $\nu_2(h)$ more explicitly than it was done in \cite{toth_95}.
Our main results are collected in the following theorem:

\begin{theorem}
\label{thm:main}
(i)
The second marginal densities $\hat\nu_2(h)$ and $\nu_2(h)$ are
\begin{align}
\label{nu2hat}
&
\hat \nu_2(h)
=
- 2 u(h) u'(h) = -\ph \left(u^2\right)(h),
\\[6pt]
\label{nu2}
&
\nu_2(h)
=
\frac{2 \cdot 6^{1/3} \sqrt{\pi}}{\Gamma(1/3)^2}
e^{-(8h^3)/9} U(1/6, 2/3; (8h^3)/9),
\end{align}
where $u:\R_+\to\R$ is the (rescaled and normed) Airy function, as defined in subsection \ref{ss:appAiry}, \eqref{airyeq}, \eqref{ubc}, and $U(1/6,2/3;\cdot)$ is the confluent hypergeometric function of the second kind (or Tricomi function), given by the integral representation \eqref{Tricomifnc} in subsection \ref{ss:confluent_hypergeometric_functions}.
\\[6pt]
(ii)
The first marginal densities $\hat\nu_1(x)$ and $\nu_1(x)$ are
\begin{align}
\label{nu1hat}
&
\hat\nu_1(x)
=
\sum_{k = 1}^{\infty}
p_k \,
\frac{\delta'_k}{2} \exp(-\delta'_k \abs{x}),
\\[6pt]
\label{nu1}
&
\nu_1(x)
=
\sum_{k = 1}^{\infty}
p_k \,
\frac{\delta'_k}{2} f_{2/3}(\delta'_k \abs{x}),
\end{align}
where $f_{2/3}$ is the Mittag-Leffler density function as defined in subsection \ref{ss:appMittag-Leffler}, the scaling factors $\delta'_k$ are the zeros of the derivative of the (rescaled and normed) Airy function $u(\cdot)$, and the coefficients $p_k$ of the convex combinations are those given in subsection \ref{ss:appAiry}, \eqref{mixweights}.
\end{theorem}

It is worth to compare the computed graphs of these density functions with histograms of the empirical distributions on $X(1)$ and $H(1)$ (see Figures \ref{histX}. and \ref{histH}.).

\begin{figure}[!h]
\label{histX}
\includegraphics[width = 12cm]{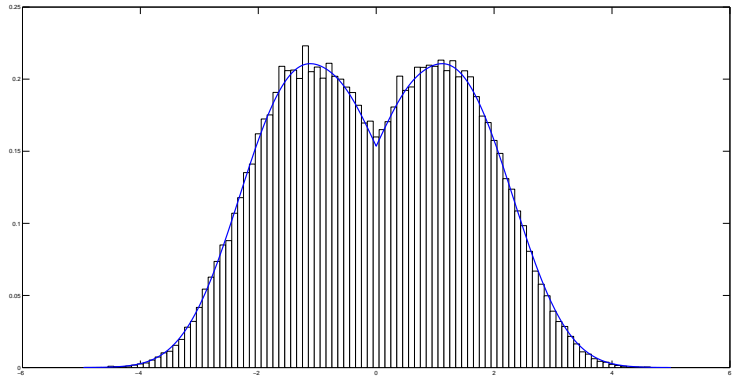}
\caption{The histogram represents a numerical simulation of the TSRM at time one $X(1)$ with a sample of $10^5$ independent realizations. Each realization is computed using the discrete ``toy-model'' introduced in \cite{toth_werner_98} with a step of $2\cdot 10^{-6}$. The plain blue line is the graph of the density of $X(1)$.}
\end{figure}
\begin{figure}[!h]
\label{histH}
\includegraphics[width = 12cm]{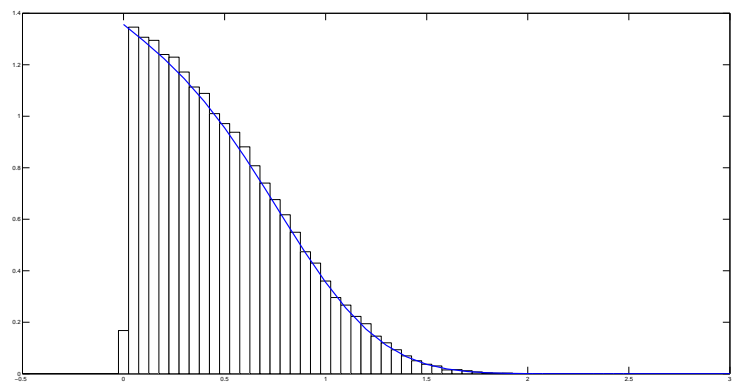}
\caption{The histogram represents a numerical simulation of the height process at time one $H(1)$ (sample of $10^5$ independent realizations using discrete ``toy-model'' with a step of $2\cdot 10^{-6}$). The plain blue line is the graph of the density of $H(1)$.}
\end{figure}

The moments of the marginal distributions collected in Theorem \ref{thm:main} are as follows
\begin{align*}
\expect{H(1)^n}
&
=
\int_0^\infty h^n\nu_2(h) dh
=
\Gamma(5/6)
\frac{(2\cdot 3^{1/3})^{-n}\,  n!}{\Gamma(n/3+1)\Gamma(n/3+5/6)}
\\[6pt]
\expect{\abs{X(1)}^{n}}
&
=
2\int_{0}^\infty x^{n} \hat\nu_1(x) dx
=
\sum_{k=1}^{\infty} p_k \frac{n!}{(\delta_k')^n \Gamma(2n/3+1)}.
\end{align*}

As a direct consequence of \eqref{nu2} and \eqref{nu1} we also obtain the precise tail asymptotics of $\nu_2(h)$ and $\nu_1(x)$ for $h\gg1$, respectively, for $\abs{x}\gg1$:

\begin{corollary}
\label{cor:tail_esrimates}
\begin{align}
\notag
&
\lim_{h\to\infty} h^{-3} \log\prob{H(1)>h} = -\frac{8}{9},
\\[6pt]
\label{Xtail}
&
\lim_{x\to\infty} x^{-3} \log\prob{X(1)>x} = -\frac{4}{27} (\delta'_1)^3.
\end{align}
\end{corollary}

\noindent In view of \eqref{xstatsimx} the tail asymptotics \eqref{Xtail} also implies
\begin{align*}
\lim_{x\to\infty} x^{-3} \log\prob{X_{\mathrm{st}}(1)>x} = -\frac{8}{27} (\delta'_1)^3.
\end{align*}
Note that this is matching with the large deviation results of \cite{dumaz_11}. More precision in the asymptotics can be derived from the expression of the densities.

\bigskip

The structure of further parts of this note is as follows: In section \ref{s:proofs} we present the proofs of the main statements. Section \ref{s:appendix} is an Appendix which contains classical ingredients: the Feynman-Kac formulas used, and collections of facts about Airy function, Mittag-Leffler distributions and confluent hypergeometric functions. We do not prove the statements collected in the Appendix but give precise reference for all.

\section{Proofs}
\label{s:proofs}

Firstly, we make a brief outline of the steps of the proof. In Section \ref{ss:preliminaries}, we factorize the density of $(X_{\theta_s},H_{\theta_s})$, i.e. $\hat \varrho(1;x,h) = \hat \nu(x,h)$ as $\hat \nu(x,h) = u(h) \varphi(x,h)$. The functions $u$ and $\varphi$ (and thus $\hat \nu$) solve differential equations we can find thanks to Feynman-Kac formula (content of Proposition 1.1). Integrating $\hat \nu(x,h)$ with respect to $x$ permits to derive results for the second marginal $\hat \nu_2$ (the density of $H_{\theta_1}$) and this is the object of Section \ref{ss:hatnu2}. Using the scaling properties and performing an inverse Laplace transform lead to Formula \eqref{nu2} for $\nu_2$ (the density of $H_1$), this is done in Section \ref{ss:nu2}. Similarly, we study the first marginal $\hat \nu_1$ in Section \ref{ss:hatnu1} and $\nu_1$ in Section \ref{ss:nu1}.

\subsection{Preliminaries}
\label{ss:preliminaries}

Recall the definitions of $T_x$, $\omega_x$ and $\omega'$ ($x\ge 0$) in Section \ref{ss:review}, \eqref{defTx}, \eqref{defwx}, and \eqref{defw'}. We write
\[
T_{x}=S + T^{1}_{x} + T^{2}_{x},
\]
where
\begin{align*}
S:=
\int_{\omega'}^0 \abs{B(y)}dy,
\qquad\qquad
T^{1}_{x}:=
\int_0^{x} \abs{B(y)} dy,
\qquad\qquad
T^{2}_{x}:=
\int_{x}^{\omega_x} \abs{B(y)} dy,
\end{align*}
(see Picture \ref{reprTx}.)

\begin{figure}[!h]
\label{reprTx}
\includegraphics[width = 10cm]{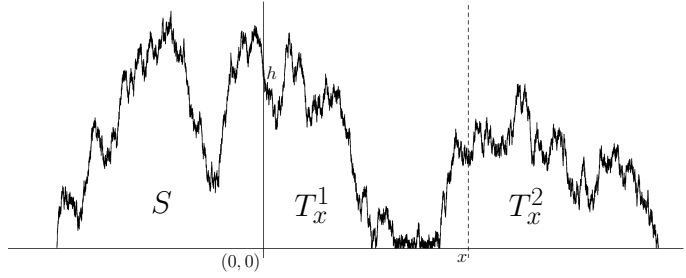}
\caption{Representation of $T_{x} = S + T^{1}_{x} + T^{2}_{x}$}
\end{figure}

The main ingredients of our investigations are the following three functions:
\begin{align}
\label{udef}
&
u(h)
:=
\condexpect{e^{-S}}{B(0)=h},
&&
h\ge0,
\\[6pt]
\label{phidef}
&
\varphi(x,h)
:=
\condexpect{e^{-T^{1}_{x} - T^{2}_{x}}}{B(0)=h},
&&
x\ge0, \,\,h\ge0,
\\[6pt]
\notag
&
w(x)
:= \varphi(x,0) =
\condexpect{e^{-T^{1}_{x} - T^{2}_{x}}}{B(0)=0}
&&
x\ge0.
\end{align}
We also define the Laplace transforms (in the variable $x\ge0$).
\begin{align}
\label{varphitilde}
\tilde\varphi(\lambda,h)
&:=
\int_0^\infty e^{-\lambda x} \varphi(x,h) dx,
\\[6pt]
\notag
\tilde w(\lambda)
&:=
\int_0^\infty e^{-\lambda x} w(x) dx.
\end{align}

\begin{remark}
 Note that this is a different type of Laplace transform than the one we used for the time variable. That is why we use different notations (hat for the time transform, tilde for this last one).
\end{remark}

In the following Proposition we collect the basic starting identities. All these follow from straightforward applications of "infinitesimal conditioning" and/or Feynman-Kac formula. By "infinitesimal conditioning" we mean using (strong) Markov property and infinitesimal generator.

\begin{proposition}
\label{prop:prelim}
(i)
The function $u:\R_+\to[0,1]$ defined in \eqref{udef} is the unique bounded solution of the boundary value problem
\begin{align}
\label{udebc}
\ph^2u(h)=2hu(h),
\qquad
u(0)=1.
\end{align}
That is: It is exactly the normalized Airy function defined and treated in subsection \ref{ss:appAiry}, restricted to $h\in[0,\infty)$.
\\[6pt]
(ii)
The function $\varphi:\R_+\times\R_+\to[0,1]$ defined in \eqref{phidef} is the unique bounded solution of the parabolic Cauchy problem
\begin{align}
\label{phipde}
\px \varphi(x,h)
=
\frac12\ph^2 \varphi(x,h) - h \varphi (x,h),
\end{align}
with initial and boundary conditions
\begin{align}
\label{phiicbc}
\mathrm{IC:}
\quad\varphi(0,h)=u(h),
\qquad
\mathrm{BC:}
\quad
\partial_h\varphi(x,0)=0, \,\,\,\text{ for } x>0.
\end{align}
(iii)
The function $\hat \nu:\R_+\times\R_+\to[0,1]$ is the unique bounded solution of the parabolic Cauchy problem
\begin{align}
\label{nuhatpde}
\px\hat\nu(x,h)
=
\frac12\ph\left( u(h)^2 \ph \left( u(h)^{-2}\hat\nu(x,h)\right)\right),
\end{align}
with initial and boundary conditions
\begin{align}
\label{nuhaticbc}
\mathrm{IC:}
\quad\hat\nu(0,h)=u(h)^2,
\qquad
\mathrm{BC:}
\quad
\hat\nu(x,0)=w(x)\,\,\, \text{ for } x>0.
\end{align}

\end{proposition}

\begin{proof}[Proof of Proposition \ref{prop:prelim}] \

\noindent
(i)
\eqref{udebc} is classical see for instance formula 2.8.1
p.~167 in \cite{borodin_salminen_02}. It follows from usual ``infinitesimal conditioning''.

\noindent
(ii)
By strong Markov property of the Brownian motion
\begin{align*}
 \varphi(x,h)
=
\condexpect{e^{-\int_0^x \abs{B(y)} dy} u(\abs{B(x)})}{B(0)=h},
\end{align*}
and thus, we apply the Feynman-Kac formula from  subsection \ref{ss:appFK}, with $V(h) = \abs{h}$ and $f(h)=u(\abs{h})$. Since for all $x>0$, $h\mapsto\varphi(x,h)$ is even and smooth, the boundary condition in \eqref{phiicbc} follows.

\noindent
(iii)
Note that
\begin{align}
\label{nu=utimesphi}
\hat \nu(x,h)=u(h)\varphi(x,h),
\end{align}
and use \eqref{udebc},\eqref{phipde} and  \eqref{phiicbc}.

\end{proof}

\subsection{Computation of $\hat\nu_2$}
\label{ss:hatnu2}

Since
\begin{align*}
\hat\nu_2(h)=2\int_0^\infty \hat\nu(x,h) dx,
\end{align*}
from \eqref{nuhatpde} and \eqref{nuhaticbc} we readily get
\begin{align}
\label{nu2hatde}
\big( u(h)^2\big( u(h)^{-2}\hat\nu_2(h)\big)'\big)'
=-4u(h)^2.
\end{align}
Using the following identities (that directly come from integration by parts and the differential equation of $u$):
\begin{align*}
& \int_0^h u(\chi)^2 d\chi = h u(h)^2 - u'(h)^2/2 + u'(0)^2/2,\\
& \int_0^h u(\chi)^{-2} u'(\chi)^2 d\chi = -u'(h)/u(h) + u'(0) + h^2,
\end{align*}
we deduce that the general solution of \eqref{nu2hatde} is
\begin{align}
\label{gensolhatnu2}
 -2u(h)u'(h) + C_1 u(h)^2 + C_2 u(h)^2 \int_0^h u(\chi)^{-2} d\chi.
\end{align}
From the asymptotics \eqref{airyasymp} it follows that
\[
u(h)^2 \int_0^h u(\chi)^{-2} d\chi \asymp h^{-1/2},
\quad\text{ as } h\to\infty.
\]
Since, by \eqref{intnu},
\begin{align}
\label{hatnu2norm}
\int_0^\infty \hat\nu_2 (h) dh =1,
\end{align}
we must choose $C_2=0$. (Otherwise $\nu_2$ wouldn't be integrable.) Finally, due to \eqref{hatnu2norm} again and the fact that $-2u u'$ already integrates to $1$, we have $C_1=0$, too. The proof of \eqref{nu2hat} is completed.

\begin{remark}
 The interchange between integration and differentiation used to obtain \eqref{nu2hatde} is standard. One can justify it properly by taking the integral over a compact interval of the type $[0,x]$ first and deduce an expression for this function of a similar kind of \eqref{gensolhatnu2}. Taking the limit $x \to \infty$ gives the result.
\end{remark}

\subsection{Computation of $\nu_2$}
\label{ss:nu2}

Due to \eqref{rholap} and \eqref{rhoscale} $\nu_2(h)$ and $\hat \nu_2(h)$ are related by
\begin{align}
\label{nu2trafo}
\hat\nu_2(h) = \int_0^\infty e^{-t} t^{-1/3} \nu_2(t^{-1/3} h) dt.
\end{align}
The goal of the present subsection is inverting this integral transform.

By straightforward computations -- performing a change of variables and an integration by parts -- form \eqref{nu2hat} and \eqref{nu2trafo} we obtain
\begin{align*}
u^2(z^{1/3})
=
\frac13\int_0^\infty e^{-sz} \nu_2(s^{-1/3}) s^{-4/3} ds,
\end{align*}
and hence
\begin{align}
\label{nuinterm}
\nu_2(h) = 3 h^{-4} (f*f) (h^{-3}),
\end{align}
where $f$ is the inverse Laplace transform of the function $z\mapsto u(z^{1/3})$:
\begin{align*}
u(z^{1/3}) = \int_0^\infty e^{-h s}f(s) ds,
\end{align*}
and $*$ stands for convolution on $[0, \infty)$.

In the following computations we denote generically by $C$ a multiplicative constant. The value of this may change from line to line.  These constants could be written explicitly in each step, but this wouldn't be much illuminating.  The value of the norming constant in \eqref{nu2} will be identified at the very end of the computations.

\noindent The function $f$ is explicitly known, see \cite{kearney_majumdar_05}:
\begin{align*}
f(s)
=
C s^{-4/3} e^{-2/(9s)}.
\end{align*}
Thus,
\begin{align}
\notag
f*f(s)
&=
C \int_0^s \big(t(s-t)\big)^{-4/3} e^{-(2s)/(9t(s-t))} dt
\\[6pt]
\notag
&=
C s^{-5/3} \int_0^1 \big(t(1-t)\big)^{-4/3} e^{-2/(9s t(1-t))} dt
\\[6pt]
\notag
&=
C s^{-5/3} e^{-8/(9s)}\int_0^\infty (1+z)^{-1/6} z^{-1/2} e^{-8z/(9s)} dz
\\[6pt]
\notag
&=
C s^{-5/3} e^{-8/(9s)} U(1/2, 4/3; 8/(9s)).
\\[6pt]
\label{f*f}
&=
C s^{-4/3} e^{-8/(9s)} U(1/6, 2/3; 8/(9s)).
\end{align}
In the third line we perform the change of integration variable $(t(1-t))^{-1}=: 4(1+z)$. Finally, in the last step we apply Kummer's identity \eqref{Kummer's_identity}.

From \eqref{nuinterm} and \eqref{f*f} we readily get
\begin{align*}
\nu_2(h)
=
C  e^{-(8h^3)/9} U(1/6, 2/3; (8h^3)/9).
\end{align*}
\noindent Identification of the norming constant $C$ in the last expression follows from \eqref{nu2hat}, \eqref{uprime0} and \eqref{nu2trafo}:
\begin{align*}
2^{4/3} 3^{1/3} \Gamma(2/3)\Gamma(1/3)^{-1}
&=
\hat\nu_2(0)
=
\Gamma(2/3)\nu_2(0)
\\[6pt]
&
=
C \Gamma(2/3)
U(1/6, 2/3; 0)
=
C
\Gamma(2/3) \Gamma(1/3) /\Gamma(1/2).
\end{align*}
\noindent Hence \eqref{nu2}.

\subsection{Computation of $\hat\nu_1$}
\label{ss:hatnu1}

Recall the definition of $\hat \nu_1$:
\begin{align*}
\hat\nu_1(x)=\int_0^\infty \hat\nu(x,h) dh.
\end{align*}
\noindent 
We examine the case $x >0$ (the function $\hat \nu_1$ is even). Using \eqref{nu=utimesphi} we readily obtain
\begin{align}
\label{pxnuhat1}
\px\hat\nu_1(x)
&= \frac12 \int_0^{\infty} \ph \left(u(\chi) \ph \varphi(x,\chi) - u'(\chi) \varphi(x,\chi)\right) d\chi \notag\\
&= \frac12 \left(u'(0) w(x) - u(0)\ph\varphi(x,0)\right).
\end{align}

\begin{remark}
 Again, the interchange between integration and differentiation is classical. Indeed, for every $x \ge 0$, $\ph \varphi(x,h) \to 0$ when $h \to \infty$. This is the case because on one hand $\ph \varphi(x,h)$ admits a finite or infinite limit when $h \to \infty$ (integrating \eqref{phipde} with respect to $h$, one can see that the terms involved admit finite or infinite limits -- recall $\px \varphi(x,h)$ is negative) and on the other hand, its integral is finite. One can then integrate $\px \hat\nu(x,h)$ first on $[0,h]$ and obtain an expression for $\int_0^h u(\chi)\varphi(x,\chi) d\chi$. Taking $h \to \infty$ and using $\lim_{h \to \infty}\ph \varphi(x,h) = 0$ rigorously prove \eqref{pxnuhat1}.
\end{remark}

Notice that we have $\ph\varphi(x,0) = 0$ for every $x$ (see \eqref{phiicbc}). The following lemma will serve to compute $w(x)$ on the right hand side of \eqref{pxnuhat1}.

\begin{lemma}
\label{lem:tildephi}
The function $\tilde\varphi:\R_+\times\R_+\to\R$ defined in \eqref{varphitilde} is expressed as
\begin{align}
\label{varphitildeexpr}
\tilde\varphi(\lambda, h)
=
\frac{-1}{u(\lambda)u'(\lambda)}
\Big\{
\,\,\,
&
u(\lambda+h) \int_0^\infty u(\lambda+\chi) u(\chi) d\chi
\\[6pt]
\notag
&
+ u(\lambda+h) \int_0^h v_\lambda(\lambda+\chi) u(\chi) d\chi
\\[6pt]
\notag
&
+v_\lambda(\lambda+h) \int_h^\infty u(\lambda+\chi) u(\chi) d\chi
\,\,\,
\Big\},
\end{align}
where $v_\lambda:\R\to\R$ is the solution of the Airy equation \eqref{airyeq} with initial conditions
\begin{align}
\label{vbc}
v_\lambda(\lambda) = u(\lambda),
\qquad
v'_\lambda(\lambda) = -u'(\lambda).
\end{align}
\end{lemma}

\begin{proof}[Proof of Lemma \ref{lem:tildephi}.] \
Note first that $\varphi(x,h)$ is exactly of the form given on the right hand side of \eqref{FKg}, with the choice $V(z)=\abs{z}$ and $f(z)=u(\abs{z})$. Thus, $\tilde\varphi(\lambda, h)$ is expressed as \eqref{FKlamsln} with this choice of $V(\cdot)$ and $f(\cdot)$.

Let $v_\lambda:\R\to\R$ be the solution of \eqref{airyeq} with boundary conditions \eqref{vbc} and
\begin{align*}
&
\phi_+(\lambda, h)
:=
u(\lambda+h)\ind{h\ge0} + v_\lambda (\lambda-h) \ind{h< 0},
\\[6pt]
&
\phi_-(\lambda, h)
:=
\phi_+(\lambda, -h).
\end{align*}
Then $\phi_\pm(\lambda, \cdot)$ are exactly the unique solutions of \eqref{FKlamaux}, with $V(h)=\abs{h}$, such that  $\lim_{h\to\pm\infty}\phi_{\pm}(\lambda,h)=0$.
Applying \eqref{FKlamsln} with $f(h) = u(\abs{h})$ and $V(h)=\abs{h}$ we readily obtain \eqref{varphitildeexpr}.
\end{proof}

\noindent Using now \eqref{varphitildeexpr},  we obtain
\begin{align}
\notag
\tilde w(\lambda)
=
\tilde \varphi (\lambda,0)
&
=
\frac{-2}{u'(\lambda)}
\int_0^\infty u(\lambda+\chi) u(\chi) d\chi
\\[6pt]
\notag
&
=
\frac{-1}{\lambda u'(\lambda)}
\int_0^\infty \big(u''(\lambda+\chi) u(\chi) - u(\lambda+\chi) u''(\chi)  \big)d\chi
\\[6pt]
&
=
\frac{-1}{\lambda u'(\lambda)}
\int_0^\infty \big(u'(\lambda+\chi) u(\chi) - u(\lambda+\chi) u'(\chi)  \big)'d\chi \notag
\\[6pt]
\label{wtilde}
&
=
\frac{u'(\lambda) - u'(0)u(\lambda)}{\lambda u'(\lambda)}.
\end{align}
\noindent In the first step we have used \eqref{varphitildeexpr}, in the second step the Airy equation \eqref{airyeq}.

\medskip

\noindent Now, using the key identity \eqref{keyidentity} and the trace formula \eqref{traceformulas}, from \eqref{wtilde} we derive
\begin{align*}
\tilde w(\lambda)
=
\frac{\abs{u'(0)}}{2}\sum_{k=1}^\infty \left(\delta'_k\right)^{-2} \frac{1}{\delta'_k+\lambda},
\end{align*}
and hence
\begin{align}
\label{w(x)}
w(x)=\frac{\abs{u'(0)}}{2}\sum_{k=1}^\infty \left(\delta'_k\right)^{-2} e^{-\delta'_k x}.
\end{align}
Finally, \eqref{pxnuhat1}, \eqref{nuhaticbc} and \eqref{w(x)} yield \eqref{nu1hat}.

\subsection{Computation of $\nu_1$}
\label{ss:nu1}

Due to \eqref{rholap} and \eqref{rhoscale} $\nu_1(x)$ and $\hat \nu_1(x)$ are related by
\begin{align*}
\hat\nu_1(x) = \int_0^\infty e^{-t} t^{-2/3} \nu_1(t^{-2/3} x) dt.
\end{align*}
The goal of the present subsection is inverting this integral transform. Since $\hat\nu_1(x)$ is convex combination of scaled exponentials, we only have to compute the inverse transform of the exponential density. Assume
\begin{align*}
e^{-x} = \int_0^\infty e^{-t} t^{-2/3} \phi(t^{-2/3} x) dt,
\qquad x\ge0.
\end{align*}
Writing the integer moments of both sides, by elementary computations we obtain
\begin{align*}
\int_0^\infty \phi(x) x^n dx = \frac {n!}{\Gamma(2n/3 + 1)}.
\end{align*}
This identifies $\phi$ as the Mittag-Leffler density $f_{2/3}$ from \eqref{MLdensity} and \eqref{ML2/3asTricomifunction} thanks to \eqref{MLmoments}. Hence \eqref{nu1}.

\section{Appendix:}
\label{s:appendix}

\subsection{Feynman-Kac formulas}
\label{ss:appFK}

According to the present context we formulate the Feynman-Kac formulas only for the one-dimensional Brownian motion. For the general theory of Feynman-Kac formulas see e.g. \cite{kac_49}, \cite{karatzas_shreve_91}, or \cite{simon_05}.

Let $y\mapsto B(y)$ be a standard 1-dimensional Brownian motion which starts from level $B(0)=h\in\R$, and $V:\R\to\R_+$ and $f:\R\to\R$ be continuous functions. Assume that $f$ is also bounded. Define
\begin{align}
\label{FKg}
&
g(x,h):=
\condexpect{e^{-\int_0^x V(B(y)) dy} f(B(x))}{B(0)=h},
&&
x\ge0,
\\[6pt]
\notag
&
\tilde g(\lambda,h):= \int_0^\infty e^{-\lambda x} g(x,h) dx,
&&
\lambda>0.
\end{align}

\begin{theorem*}
[The Feynman-Kac formula for 1-d Brownian motion] \
\\[6pt]
(i)
$(x,h)\mapsto g(x,h)$ is the unique bounded solution of the parabolic Cauchy problem
\begin{align}
\notag
\px g(x,h) = \frac12 \ph^2 g(x,h) -V(h) g(x,h),
\qquad
g(0,h) = f(h).
\end{align}
(ii)
For $\lambda>0$ fixed, $h\mapsto \tilde g(\lambda,h)$ is the unique bounded solution of the elliptic PDE
\begin{align}
\label{FKlam}
\frac12 \ph^2 \tilde g(\lambda,h) =
\big(V(h)+\lambda\big)\tilde g(\lambda,h)  - f(h).
\end{align}
(iii)
Fix $\lambda>0$ and let $\phi_\pm(\lambda, \cdot):\R\to\R$ be the unique solutions of the ODE
\begin{align}
\label{FKlamaux}
\frac12 \ph^2 \phi(h) =
\big(V(h)+\lambda\big)\phi(h),
\end{align}
for which $\lim_{h\to\pm\infty}\phi_{\pm}(\lambda, h)=0$. (The convergence is actually at least exponentially fast.) Then the unique bounded solution of \eqref{FKlam} is expressed as:
\begin{align}
\label{FKlamsln}
\tilde g(\lambda,h)
=&
2\left(
\phi_+(\lambda,0)\phi'_-(\lambda, 0) -
\phi_-(\lambda, 0)\phi'_+(\lambda, 0)
\right)^{-1}
\\[6pt]
&
\notag
\hskip1cm
\times
\left\{
\phi_+(\lambda, h)\int_{-\infty}^h\phi_-(\lambda, \chi)f(\chi)d\chi
-
\phi_-(\lambda, h)\int_h^{\infty}\phi_+(\lambda, \chi)f(\chi)d\chi
\right\}.
\end{align}
\end{theorem*}

\medskip

\noindent For a full proof of these statements see e.g. \cite{janson_2007}.

\subsection{Airy functions}
\label{ss:appAiry}

In the present Appendix we collect the \emph{necessary minimum} information about Airy function needed and used in this paper. For an exhaustive treatment of Airy functions see \cite{abramowitz_stegun_64}, \cite{flajolet_louchard_01}, \cite{kearney_majumdar_05}, \cite{reid_95}.

\medskip

\noindent Consider the second order linear ODE for $f:\R\to\R$:
\begin{align}
\label{airyeq}
f''(h)=2h f(h),
\end{align}
commonly called Airy's equation. We denote by $u:\R\to\R$ the unique solution of \eqref{airyeq} with boundary conditions
\begin{align}
\label{ubc}
u(0)=1,
\qquad
\lim_{h\to\infty} u(h) = 0.
\end{align}
In terms of the conventionally defined Airy function of the first kind $\Ai(\cdot)$ we have
\begin{align*}
u(h)= 3^{2/3}\Gamma(2/3)\Ai(2^{1/3}h),
\end{align*}
and we also have
\begin{align}
\label{uprime0}
&
u'(0)=-\frac{6^{1/3} \Gamma(2/3)}{\Gamma(1/3)},
\\[6pt]
\label{airyasymp}
&
u(h)\sim \frac{3^{2/3} \Gamma(2/3)}{2^{13/12}\sqrt{\pi}} h^{-1/4} \exp\left(-(2^{3/2}/3) h^{3/2}\right),
&&
\text{ as } h\to\infty
\\[6pt]
\notag
&
u(h)\sim \frac{3^{2/3} \Gamma(2/3)}{2^{1/12}\sqrt{\pi}} \abs{h}^{-1/4} \sin \left((2^{3/2}/3) \abs{h}^{3/2} + \pi/4\right),
&&
\text{ as } h\to-\infty.
\end{align}
The function $u:\R\to\R$ is convex and decreasing on $[0,\infty)$ and oscillates indefinitely on $(-\infty,0]$. The function $u$ extends as an entire function to $\C\ni z\mapsto u(z)\in \C$.

\begin{figure}
\includegraphics[width = 7cm]{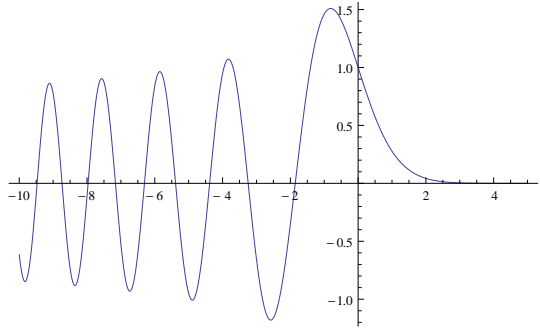}
 \caption{The normalized Airy function $u$}
\end{figure}

\medskip

\noindent We denote by $\delta'_k$, $k=1,2, \dots$ the consecutive zeros of $h\mapsto u'(-h)$. It is known that
\[
0<\delta'_1<\delta'_2<\dots<\delta'_k<\dots,
\]
and their asymptotics, as $k\to\infty$, are
\[
\delta'_k\sim\frac12 (3\pi k)^{2/3}.
\]
(Finer asymptotics are also available but not needed for our purposes in this paper.)

\medskip

\noindent The following key identity holds
\begin{align}
\label{keyidentity}
&
\frac{u''(z)}{u'(z)}
=
-\sum_{k=1}^\infty \frac{1}{\delta'_k}\cdot\frac{z}{\delta'_k+z},
&&
z\notin\{-\delta'_k: k=1,2\dots\}.
\end{align}
It is easily seen that the sum on the right hand side is absolutely convergent and the two sides have simple poles at $-\delta'_k$, $k=1,2,\dots$, with the same residues. Thus, the two sides can differ in an entire function only. For a full proof of this formula see \cite{flajolet_louchard_01}.

From \eqref{keyidentity} an infinite family of \emph{trace formulas} for $\sum_{k=1}^\infty (\delta'_k)^{-n}$ ($n \ge 2$) are expressed in terms of $u'(0)$. We will use only the first three of these:
\begin{align}
\label{traceformulas}
\sum_{k=1}^\infty (\delta'_k)^{-2}=\frac{-2}{u'(0)},
\qquad\qquad
\sum_{k=1}^\infty (\delta'_k)^{-3}=2,
\qquad\qquad
\sum_{k=1}^\infty (\delta'_k)^{-4}=\frac{2}{u'(0)^2}.
\end{align}
These are easily checked by differentiating both sides of \eqref{traceformulas} at $z=0$ and using \eqref{airyeq}.
We will denote
\begin{align}
\label{mixweights}
p_k:= \frac{u'(0)^2}{2}(\delta'_k)^{-4}.
\end{align}
Note that $\sum_{k=1}^\infty p_k = 1$.

\subsection{Mittag-Leffler distributions}
\label{ss:appMittag-Leffler}

Following the terminology of \cite{feller_49} and \cite{darling_kac_57}, we will call Mittag-Leffler distribution of index $\alpha\in[0,1]$ the probability distribution $F_\alpha: \R_+ \to [0,1]$  whose moment generating function (Laplace transform) is
\begin{align}
\label{ML}
\int_0^\infty e^{- yx} dF_\alpha(x)
=
\sum_{k=0}^\infty \frac{(-y)^k} {\Gamma(\alpha k +1)}
=:
E_\alpha (-y).
\end{align}
The function $y\mapsto E_\alpha(y)$ defined by the power series in \eqref{ML} is the Mittag-Leffler function of index $\alpha\in[0,1]$. It is proved in \cite{pollard_48} -- where reference is also made to unpublished alternative proof of W. Feller -- that, for $\alpha\in[0,1]$,  $[0,\infty)\ni y\mapsto E_{\alpha}(-y)$ is indeed completely monotone, and thus the Laplace transform of a probability density function on $\R_+$.

\medskip

\noindent The moments of the Mittag-Leffler distributions are
\begin{align}
\label{MLmoments}
\int_0^\infty x^m dF_\alpha(x)
=
\frac{m!}{\Gamma(\alpha m +1)}.
\end{align}
$\alpha=0, \frac12, 1$ are special cases:
\begin{align*}
F_0(x)=1-e^{-x},
\qquad\qquad
F_{1/2}(x)=\sqrt{\frac{2}{\pi}} \int_0^x e^{-y^2/2}dy,
\qquad\qquad
F_1(x)=\ind{x>1}.
\end{align*}

For $\alpha\in[0,1)$ the distribution function $F_\alpha$ is smooth and the following power series expansion holds for the density function $f_\alpha(x):=F'_\alpha(x)$, see \cite{pollard_48}:
\begin{align}
\label{MLdensity}
f_\alpha(x)=
\frac{1}{\pi} \sum_{k=0}^\infty
\frac{\sin((k+1)\alpha\pi) \Gamma((k+1)\alpha)}{k!} (-x)^{k}
\end{align}
This power series defines actually an entire function on $\C$. See the picture below for a representation of those functions with $\alpha = 0, 1/2$ and $2/3$.

\begin{figure}[!h]
\includegraphics[width = 10cm]{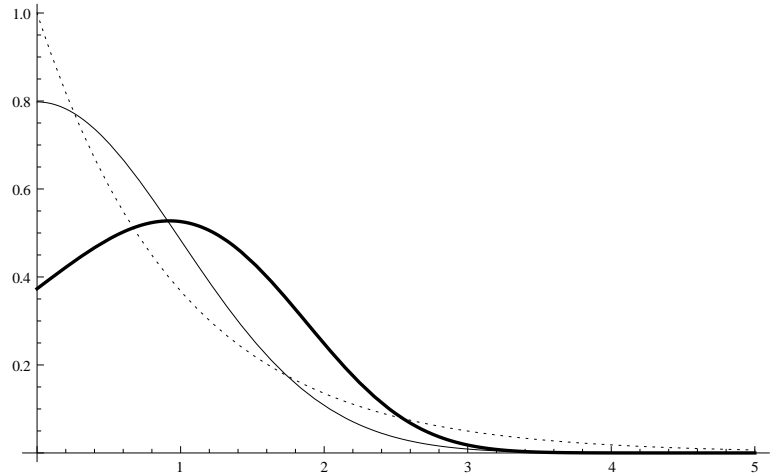}
\caption{Mittag-Leffler density functions for $\alpha = 0$ (dotted), $\alpha = 1/2$ (plain) and $\alpha = 2/3$ (thick)}
\end{figure}

\noindent We are interested in the case $\alpha=2/3$. This particular case is also expressed in terms of a confluent hypergeometric function, see \eqref{ML2/3asTricomifunction} below.

\subsection{Confluent hypergeometric functions}
\label{ss:confluent_hypergeometric_functions}

See \cite{gradstein_ryzhik_46}, \cite{abramowitz_stegun_64} for details on confluent hypergeometric functions. We collect here only a few facts needed for our purposes.

Let $a>0$ and $b\in\R$ be fixed parameters. We define the \emph{confluent hypergeometric functions of the second kind}, or \emph{Tricomi functions}
\[
U(a,b;\cdot):\R_+\mapsto \R
\]
by the integrals
\begin{align}
\label{Tricomifnc}
U(a,b;z)
=
\frac{1}{\Gamma(a)}
\int_0^\infty e^{-zs}s^{a-1}(1+s)^{b-a-1} ds.
\end{align}
(For more general definitions see \cite{gradstein_ryzhik_46}, \cite{abramowitz_stegun_64}.) In our formulas the particular cases $U(1/6,2/3;\cdot)$,  $U(1/2,4/3;\cdot)$, and $U(1/6,4/3;\cdot)$ will occur.

\emph{Kummer's identity}
\begin{align}
\label{Kummer's_identity}
U(a,b;z)= z^{1-b} U(1+a-b, 2-b, z)
\end{align}
holds if $b<a+1$.

The Mittag-Leffler density function $f_{2/3}$, which appears in \eqref{nu1} is also expressed in terms of $U(1/6, 4/3;\cdot)$ as follows:
\begin{align}
\label{ML2/3asTricomifunction}
f_{2/3} (z)
=
\frac{2^{1/3}}{\sqrt{3 \pi}}
z \,e^{-(4z^3)/27} U(1/6, 4/3;(4z^3)/27).
\end{align}
Computing the moments of the right hand side of \eqref{ML2/3asTricomifunction} we obtain the expressions \eqref{MLmoments}.

\medskip

\noindent \textbf{Acknowledgements:}
The following grant supports are acknowledged: French-Hun\-ga\-ri\-an
Balaton/PHC grant 19482NA for mobility support; T\'AMOP -
4.2.2.B-10/1--2010-0009 supporting research at the Graduate School of
Mathematics TU Budapest; OTKA K100473 partially supporting BT's research.

\vskip4cm

\noindent
{\sc Affiliation of authors:}
\\[10pt]
\hbox{
\vbox
{\hsize=15mm
\noindent 
LD:
\\
}
\vbox
{\hsize=12cm
\noindent
DMA - \'Ecole Normale Sup\'erieure, Paris
\\ 
email: {\tt Laure.Dumaz@ens.fr}}
}
\\[10pt]
\hbox{
\vbox
{\hsize=15mm
\noindent 
BT:
\\ \\
}
\vbox
{\hsize=12cm
\noindent
Institute of Mathematics, Budapest University of Technology, and 
\\
School of Mathematics, University of Bristol
\\
email: {\tt balint@math.bme.hu, balint.toth@bristol.ac.uk}}
}

\end{document}